\documentclass[12pt,a4paper]{article}
 \usepackage{amsmath,amstext,amssymb,amscd}
  \usepackage[english]{babel}
 \usepackage{graphicx}

\newcommand{\Xcomment}[1]{}

\newtheorem{theorem}{Theorem}[section]
\newtheorem{lemma}[theorem]{Lemma}

\newtheorem{prop}[theorem]{Proposition}

\makeatletter \@addtoreset{equation}{section} \makeatother

\def\Rset{{\mathbb R}}

\def\Ascr{{\cal A}}

\def\Mscr{{\cal M}}

\def\Qscr{{\cal Q}}

\usepackage{mathtools}

\newenvironment{proof}{\noindent{\bf Proof}~}%
{\hfill$\qed$\medskip}

\newenvironment{numitem1}{\refstepcounter{equation}\begin{enumerate}%
\item[(\thesection.\arabic{equation})]}{\end{enumerate}}

\def\qed{ \ \vrule width.1cm height.3cm depth0cm}

\begin{document}

\title{Higher Bruhat orders of types B and C }

 \author{Vladimir I.~Danilov
 \thanks{Central Institute of Economics and
Mathematics of the RAS, 47, Nakhimovskii Prospect, 117418 Moscow, Russia;
email: danilov@cemi.rssi.ru.}
 \and
Alexander V.~Karzanov
\thanks{Central Institute of Economics and Mathematics of
the RAS, 47, Nakhimovskii Prospect, 117418 Moscow, Russia; email:
akarzanov7@gmail.com. }
  \and
Gleb A.~Koshevoy
\thanks{Institute for Information Transmission Problems of
the RAS, 19, Bol'shoi Karetnyi per., 127051 Moscow, Russia, and HSE University;
email:koshevoyga@gmail.com. Supported in part by grant 22-41-02028 from the Russian Science Foundation.}
  }

\date{}

\maketitle

\renewcommand{\abstractname}{Abstract}

\begin{abstract}
\noindent We propose versions of higher Bruhat orders for types $B$ and $C$. This is based on a theory of higher Bruhat orders of type~A and their geometric interpretations (due to Manin--Shekhtman, Voevodskii--Kapranov, and Ziegler), and on our study of the so-called symmetric cubillages of cyclic zonotopes.
\end{abstract}

\section{Introduction}\label{intr}

In their study of Zamolodchikov's equations and a multi-dimensional generalization of Yang-Baxter's ones, Manin and Shekhtman~\cite{MS} introduced the notion of \emph{higher Bruhat orders} $B(n,d)$, which are posets (partially ordered sets) defined for any positive integers $n$ and $d\le n$. When $d=1$, this turns into the classical \emph{weak (Bruhat) order} on the set of permutations on the $n$-element set $[n]:=\{1,\ldots,n\}$ (usually equipped with the natural order $1<2<\cdots<n$). In general, $B(n,d+1)$ is obtained by a factorization from the set of maximal chains in $B(n,d)$. Note that the notion of a weak order can be defined for any Coxeter group, and the higher Bruhat orders due to Manin and Shekhtman can be attributed as being of type $A$. 

In this paper, continuing our study of symmetric structures started in \cite{DKK2, DKK3}, we propose versions of higher Bruhat orders for types $B$ and $C$.

The main idea of our construction is briefly as follows. Kapranov and Voevodsky \cite{KV} showed that one can express $B(n,d)$ as the poset ${\mathbf Q}(n,d)$ of fine zonotopal tilings, or \emph{cubillages} (in terminology of~\cite{KV}), in the \emph{cyclic zonotope} $Z(n, d)$. Here the order structure is given via the so-called \emph{flips}, transformations that turn one cubillage into another. Since Coxeter groups of types $B$ and $C$ can be understood as groups of some symmetric permutations (see~\cite[Sect.~6]{eln}), it is natural to try to construct higher Bruhat orders as posets on symmetric cubillages. Here the symmetry is given by the involution $i \mapsto i^\circ :=n+1-i$ on the set $[n]$, which we call \emph{the color involution}.

This paper is organized as follows. In Section~\ref{sec2} we recall basic definitions and review some facts on cubillages and related objects. Section~\ref{sec3} defines the set $\mathbf{SQ}(n,d)$ of symmetric cubillages in the cyclic zonotope $Z (n,d)$, Section~\ref{sec4} is devoted to certain geometric aspects related to such cubillages, and Section \ref{sec5} introduces the notion of (generalized) flips on them. This allows us to define the structure of an acyclic  digraph on $\mathbf{SQ}(n,d)$, giving rise to a poset on it. The structures of such posets are viewed differently for the even and odd cases of $n$. When $n$ is even, $n=2m$, the poset $\mathbf{SQ}(n,d)$ is attributed as a higher Bruhat order of \emph{type C} and denoted as $\mathbf{C}(m,d)$, whereas for $n$ odd, $n=2m+1$, it is attributed as that of \emph{type~B} and denoted as $\mathbf{B}(m,d)$. Note that for $d=1$, these posets coincide and represent the same weak order (simultaneously of types $C$ and $B$).

Each of the posets $\mathbf{C}(m,d)$ and $\mathbf{B}(m,d)$  has as a minimal element the so-called \emph{standard cubillage} $\mathcal Q^{st}(n',d)$ (where $n'$ is $2m$ or $2m+1$). We conjecture that this is the \emph{unique} minimal element in the poset (Conjecture~1 in Section~\ref{sec5}). Symmetrically, each of $\mathbf{C}(m,d)$ and $\mathbf{B}(m,d)$  has as a maximal element the so-called \emph{anti-standard cubillage} $\mathcal Q^{ant}(n',d)$, and accordingly, this cubillage is conjectured to be the unique maximal element there. We  verify this conjecture in three special cases expounded in propositions of  Section~\ref{sec6}. In particular, it is valid when $n$ is even and $d$ is odd. 

Sections \ref{sec7} and~\ref{sec8} contain appealing examples of the above posets and construct a ``reduction'' morphism $\textbf{B} (m, d)\to \textbf{C}(m,d)$, which looks rather nontrivial. We prove the connectedness of the fibers in this morphism. Section~\ref{sec9} is devoted to a relation to the higher Bruhat orders of type~A. Here we construct morphisms of the form $\textbf{C}(m,d) \to \textbf{A} (m,\lfloor d/2 \rfloor)$ and $\textbf{B} (m,d) \to \textbf{A} (m+1,\lfloor d/2\rfloor)$.

The concluding Section \ref{sec10} discusses maximal chains in the posets $\mathbf{C}(m,d)$ and $\mathbf{B}(m,d)$. Here we observe that the cases of $d$ even and odd behave differently. Namely, for $d$ odd, 
$\mathbf{C}(m,d+1)$ and  $\mathbf{B}(m,d+1)$ are obtained as factorizations from maximal chains in $\mathbf{C}(m,d)$ and $\mathbf{B}(m,d)$, respectively, while for $d$ even, the sets of maximal chains turn out to be rather poor. (Note that  the cases $d=1,2$ were considered in~\cite{SAV}.) Due to the former, the corresponding Zamolodchikov's equations for types $C$ and $B$ become well-motivated when the dimension is odd.

           
\section{Preliminaries. Basic facts on cubillages}\label{sec2}

In this section we give necessary definitions and recall basic facts on cubillages of cyclic zonotopes used in the paper. For details, see \cite{UMN-19}.

A \emph{zonotope} $Z([n],d)$ of dimension $d\ge 1$ (usually we will write simply $Z(n,d)$) is defined by use of $n\ge d$ \emph{guiding} vectors $\xi _1,\ldots,\xi _n$ in the space $\mathbb R ^{[d]}$ (with the unit base vectors $e_1,\ldots, e_d$) and is a polytope represented as the Minkowskii sum of segments $[0,\xi_i]$; this is formed by the points (vectors) $\sum_i a_i \xi _i$, where $0\le a_i\le 1$. We assume that in the $d\times n$ matrix formed by $\xi_1,\ldots,\xi_n$ as columns (in this order), all flag minors are strictly positive, in which case the guiding system $\xi$, as well as the zonotope, is called \emph{cyclic}. 
It is often convenient to take points $\xi _i$ located on the Veronese curve, by setting $\xi_i=(1, t_i,\ldots, t_i^{d-1})$ for $t_1<\cdots<t_n$. We interpret the elements of $[n]$ as \emph{colors}. When $n=d$, the zonotope $Z(d,d)$ (as well as any of its relevant shifts) is called a \emph{cube} for short. Using terminology of~\cite{UMN-19}, we call the zonotope $Z(d+1,d)$ a \emph{capsid}, and $Z(d+2, d)$ a \emph{barrel}.

A \emph{cubillage} is a subdivision  $\mathcal Q $ of the zonotope $Z=Z(n,d)$ into cubes, so that any two intersecting cubes share a common face and any facet ($(d-1)$-dimensional face) in the boundary of $Z$ is covered by one cube. 
We may understand $\mathcal Q $ either as the set of its cubes, or as the corresponding cubic complex. Each cube $C$ in $\mathcal Q $ is of the form $Z(D,d)$, where $D$ is a subset of $[n]$ of size $d$; we call $D$ the {\em type} of $C$. One shows that the correspondence $C\mapsto D$ establishes a bijection
between the set of cubes of a cubillage $\mathcal Q $  and the set $Gr([n],d)=\binom{[n]}{d}$ of subsets of $[n]$ of size $d$ (usually called the {\em discrete Grassmannian}). This bijection enables us to switch, when needed, from a geometric representation of cubillages to a combinatorial one.

A \emph{packet} (in terminology of~\cite{MS}) is meant to be a subset $F$  of size $d+1$ in $[n]$. It determines the family $Gr (F,d)$ in $Gr([n], d)$ consisting of $d$-element subsets of $F$. We refer to such a family as a \emph{stick} and assume that its members are ordered \emph{lexicographically}. For example, if $d=4$ and $F=12345$, then the corresponding stick consists of the ordered 5-tuple  
\[
1234 < 1235 < 1245 < 1345 < 2345.\]

\noindent(Hereinafter for a subset $\{a,b\ldots,c\}$ of $[n]$, we
use the shorter notation $ab\cdots c$.)
 \medskip

\noindent\textbf{Inversions.} 
An important representation of cubillages involves \emph{inversion systems}. In this way, each cubillage $\mathcal Q $ is associated with a certain subset $Inv (\mathcal Q )$ (of \emph{inversions}) in $Gr([n], d+1)$ which uniquely determines $\mathcal Q$. 

 \begin{figure}[htb]
\begin{center}
\includegraphics[scale=.4]{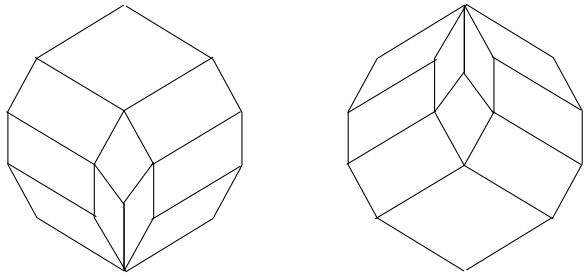}
\end{center}
\vspace{-0.4cm}
\caption{Standard (left)  and anti-standard (right) rhombus tilings of the zonogon $Z(5,2)$.}
 \label{fig1}
\end{figure}

Geometrically, the inversion system $Inv (\mathcal Q )$ can be  constructed as follows (for details, see \cite[Sect.~15]{UMN-19}). Represent a cubillage $\mathcal Q$ in $Z=Z(n,d)$ as the corresponding complex $\Mscr$ within the zonotope $\widetilde{Z}=Z(n, d+1)$ of the next dimension $d+1$ (using the natural bijection between the guiding vectors in $Z$ and $\widetilde Z$). This $\Mscr$, called an (abstract) \emph{membrane} in $ \widetilde{Z}$, divides this zonotope into two halves: before the membrane and after it (where ``before'' and ``after'' are positioned in the direction of the last base vector $e_{d+1}$ in $\mathbb R ^{d+1}$). The membrane $\mathcal M $ can always be embedded in some cubillage $\widetilde{\mathcal Q} $ of the zonotope $\widetilde{Z}$. Then $Inv(\mathcal Q )$ is just the set (system) of \emph{types} of the cubes of $\widetilde{\mathcal Q} $ that are located \emph{before} $\mathcal M $. (Although there may be many cubillages in $\widetilde{Z}$ containing  $\mathcal M $, the system $Inv(\mathcal Q )$ is defined uniquely.) For example, the empty inversion system gives the so-called \emph{standard} cubillage in $Z$; geometrically, it is represented (under the projection $\mathbb{R}^{[d+1]} \to \mathbb{R}^{[d]}$) by the membrane being the visible part (viz. the \emph{front} boundary) of  $\widetilde{Z}$. In turn, the system consisting of the entire $Gr ([n], d+1)$ gives the \emph{anti-standard} cubillage in $Z$, which is represented by the membrane being the invisible part (viz. the \emph{rear} boundary) of $ \widetilde Z$. See Fig.~\ref{fig1}.

The set $Inv (\mathcal Q )$ of inversions of a cubillage $\mathcal Q$ in $Z(n,d)$ (equivalently, the set $Inv (\Mscr)$ of inversions of the corresponding membrane $\Mscr$ in $Z(n,d+1)$) has the following important characterization:
  \begin{description} 
\item[{\emph Ziegler's condition}]\cite{zieg}: for any packet  $F$ of size $d+2$ in $[n]$, the intersection of $Inv (\mathcal Q )$ 
with the stick $Gr(F, d+1)$ is either an initial or a final interval of the stick.
  \end{description}
 
Equivalently, if we interpret the intervals in a stick as \emph{convex} subsets, then the above intersections are convex and their complements are convex as well; we will refer to such intersections as \emph{bi-convex}. This provides us with a useful representation of cubillages via bi-convex systems in $Gr([n],d+1)$.

\vspace{-0.2cm}
\begin{figure}[htb]
\begin{center}
\includegraphics[scale=.35]{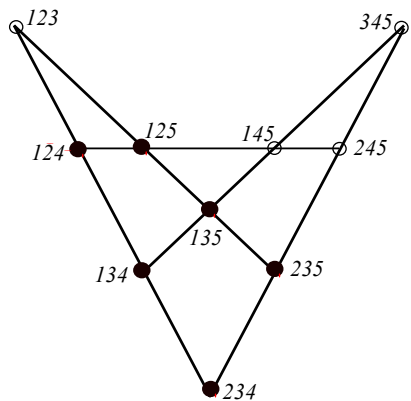}
\end{center}
\vspace{-0.4cm}
\caption{A rhombus tiling for $n=5$ represented via an inverse system (indicated by black vertices).}
\end{figure}

Note that if a system $S$ in $Gr([n], d+1)$ is bi-convex, then its complementary system $\overline{S} =Gr([n], d+1) - S $  is bi-convex as well. This gives an involution on the set of cubillages  ${\bf Q} (n,d)$, namely, $\mathcal{Q} \mapsto \overline{\mathcal{Q}}$ (geometrically, this is realized by reflecting a cubillage $\mathcal Q $ with respect  to the center of symmetry of the zonotope $Z (n,d)$).
\medskip

\noindent\textbf{Flips.} 
Next we recall facts about local transformations, or {\em  flips}, in cubillages (for details, see \cite[Sects.~8,12]{UMN-19}). As mentioned above, the simplest, after cubes, sort of zonotopes is formed by capsids 
$Z (d+1, d)$. Any capsid has exactly two cubillages: the standard and anti-standard ones (this follows from the fact that the cube $C=Z(d+1,d+1)$ has exactly two membranes: the front and rear sides of $C$). If a cubillage $\mathcal Q$ contains a capsid fragment $K$, we can replace the existing filling $\mathcal Q |_K$ in it by the other one. This local exchange is called a  \emph{flip} in the cubillage, and we say that this flip is \emph{raising} (\emph{lowering}) if it replaces the standard filling in a capsid by the anti-standard one (resp. the anti-standard by standard one). For example, in the standard tiling in Fig.~\ref{fig1} (left), one can see three capsids (viz. hexagons), all having the standard fillings, and we can make a raising flip in each of them.

Associating each raising flip with an \emph{arrow} (directed edge) between the corresponding cubillages, we get a structure of \emph{digraph} (directed graph) on the set ${\bf Q}(n,d)$. See e.g.~\cite{zieg}, where  such a digraph is depicted for ${\bf Q}(5,2)$ (it has 62 vertices).

Combinatorially, the raising flip in a capsid of type $F$ in a cubillage $\mathcal Q$ leads to increasing the set of inversions by one element, namely, by $F$ (so $Inv (\mathcal Q' )=Inv (\mathcal Q )\cup\{F\}$, where $\mathcal Q' $ is the resulting cubillage). In particular, this implies that the digraph ${\bf Q}(n,d)$ is acyclic. The transitive closure of this digraph turns  ${\bf Q} (n,d)$ into a poset (partially ordered set), which is a sample of higher Bruhat orders of  type $A$. This poset is ranked, where the rank of a cubillage $\mathcal Q $ is the cardinality of $Inv(\mathcal Q )$.
The standard cubillage is the unique source (zero-indegree, or minimal, vertex) 
of the digraph ${\bf Q}(n,d)$; symmetrically, the anti-standard cubillage is the unique sink (zero-outdegree, or maximal, vertex). 
\medskip

\textbf{Ziegler's trellises.} 
Following Ziegler \cite{zieg}, for any  $n,d$, one can attempt to draw the Grassmannian $Gr ([n],d+1)$, together with the stick system (bijective to $Gr ([n], d+2)$),  on the plane so that the sticks be depicted as straight line segments. 
We call such a diagram (when succeeded) a {\em Ziegler trellis} and denote it by $Tre(n,d)$. 
There are $d+2$ vertices on each stick, and each vertex belongs to $n-d-1$ sticks.
When considering symmetric cubillages, it is natural to draw these diagrams symmetrically. Without coming into detailed explanations, we illustrate some enlightening cases in Fig.~\ref{fig3}.

 \begin{figure}[htb]
\begin{center}
\includegraphics[scale=.7]{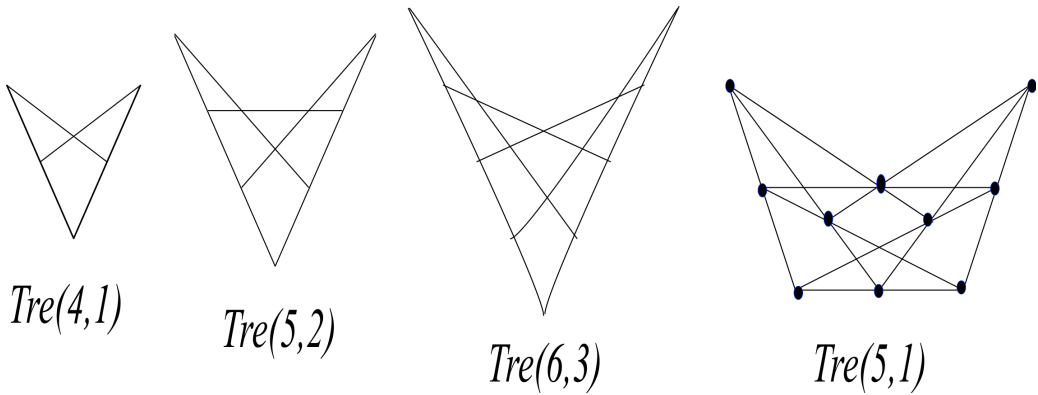}
\end{center}
\vspace{-0.4cm}
\caption{Four Ziegler's trellises.}
 \label{fig3}
\end{figure}

To specify a cubillage, we should select an appropriate (bi-convex) subset of vertices in the corresponding trellis. 


\section{Symmetric cubillages }\label{sec3}

The symmetry of cubillages can be stated in different ways. 
We consider one specific involution, which admits two (equivalent) definitions. We start with a faster (though less visual) combinatorial definition. A  geometric one is given in the next section.

Define the involution on the set $[n]$ of colors by $i \mapsto i^\circ :=n+1-i$. This is extended, in a natural way, to the involution $X\mapsto X^\circ:=\{i^\circ\colon i\in X\}$ on the subsets $X$ of $[n]$, yielding a natural involution on the Grassmannian $Gr([n],d)$ for each $d$. If $ \mathcal S $ is 
a bi-convex system in $Gr ([n],d+1)$, then the system 
$\mathcal S ^\circ =\{X\colon X^\circ \in \mathcal S \}$ is bi-convex as well, in view of the obvious identity for $F\subseteq [n]$ of size $d+2$:
$$
                             Gr(F^\circ ,d+1)=Gr(F,d+1)^\circ .
                       $$
This enables us to define the involution $\mathcal Q \mapsto \mathcal Q ^\circ $ on cubillages by use of the formula on inversions
                         $$
  Inv(\mathcal Q ^\circ       )=Inv(\mathcal Q )^\circ .
                           $$

\noindent\textbf{Definition.} A cubillage $\mathcal Q $ is called \emph{symmetric} if $\mathcal Q ^\circ =\mathcal Q $. 
\medskip


Symmetric cubillages correspond to symmetric bi-convex collections $\mathcal S $ in $Gr([n], d+1)$, that is, such that $\mathcal S =\mathcal S ^\circ $. 
Symmetric cubillages exist for any $n,d$. In particular, both the standard and anti-standard cubillages are such. The set of symmetric cubillages in $Z (n,d)$ is denoted by ${\bf SQ} (n,d)$. \medskip

\noindent\textbf{Example 1.} 
A cubillage for $d=1$ is nothing else than a sequence of $n$ different numbers in $[n]$ (a permutation). The symmetry of such a permutation  means that if a color $i$ occurs at $k$-th place, then the symmetric color $i^ \circ $ does at place  $k^\circ $. For example, suppose that color $5$ be situated before than $2$. Then the pair $\{2,5\}$ belongs to $Inv(\mathcal Q )$. By symmetry, the pair $\{2^\circ ,5^\circ \}$ belongs to $Inv(\mathcal Q )$ as well. Since $5^ \circ <2^\circ $, we obtain that $5^\circ $ is situated later than $2^\circ $.

If $n=2m$, then a symmetric permutation is equivalent to a signed permutation on the set $[m]$, namely, a permutation whose elements are endowed with sign $+$ or $ - $ (where $ + i$ ($-i$) occurring at place $k$ means that in the original permutation this place is occupied by $i$ (resp. $i^\circ=2m-i+1$)). In~\cite{hump}  the signed permutations are attributed as elements of the group $B_m$ (see also~\cite{eln,BB}). (The quantity of such permutations is $m!2^m$.) However, this group structure is beyond our consideration here, and for reasons that will be clearer later, we prefer to talk about type $C_m$, rather than $B_m$, at this point. 

If $n=2m+1$, then at the middle, $(m+1)$-th, place of a symmetric permutation, there always stays the middle color $m+1$. This implies that in order to form a symmetric permutation for $2m+1$, one should take a sign permutation on $[m]$, followed by color $m+1$, and then proceed by the symmetry. So the set (and digraph) ${\bf SQ}(2m+1,1)$ is isomorphic to ${\bf SQ}(2m, 1)$, and the corresponding poset is attributed as the \emph{weak (Bruhat) order} $B_m$.\medskip

\noindent\textbf{Example 2.} 
Fig.~\ref{fig4} illustrates two symmetric cubillages for $d=2$ (rhombus tilings).
\medskip
       \begin{figure}[htb]
\begin{center}
\includegraphics[scale=.4]{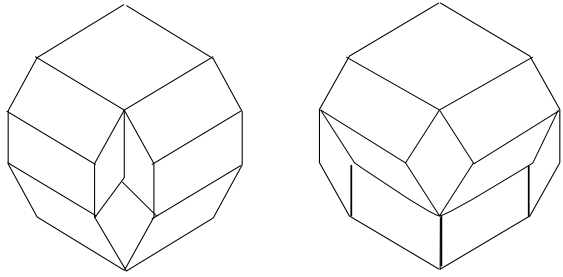}
\end{center}
\vspace{-0.4cm}
\caption{Two symmetric rhombus tilings of the zonogon $Z (5,2)$.}
 \label{fig4}
\end{figure}

\noindent\textbf{Example 3.} 
When $d=3$, it is not so easy to visualize a symmetric cubillage directly. Instead, we can use the combinatorial representation via inversions. For example, for $n=6$, one can take the two-element system $\{1234, 3456\}$ in $Gr([6], 4)$, which gives rise to a symmetric cubillage with two inversions in $Z(6,3)$. For these $n,d$, it is convenient to use the corresponding Ziegler's trellis. Then, to specify a symmetric cubillage, one should take (mark) a subset of vertices on the trellis which is bi-convex and symmetric. See Fig.~\ref{fig5} where the vertices 1234 and 3456 are marked.

 \begin{figure}[htb]
\begin{center}
\includegraphics[scale=.6]{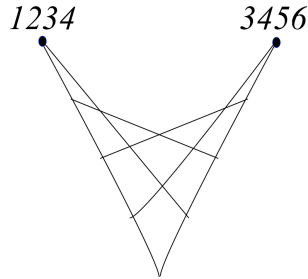}
\end{center}
\vspace{-0.4cm}
\caption{An example of symmetric cubillages of $Z(6,3)$ in terms of Ziegler's trellis.}
 \label{fig5}
\end{figure}

  \medskip
  
\noindent\textbf{Remark 1.} Acting in terms of inversion sets, one can also introduce the notion of a \emph{skew-symmetric} cubillage; this is a cubillage whose inversion system is defined to be a bi-convex collection $\mathcal S $ in $Gr([n], d+1)$ such that $\mathcal S ^\circ =Gr([n],d+1) - \mathcal S $. Such cubillages are encountered in~\cite{DKK3} (which studies a variety of cubillages $\Qscr$ whose vertex set is symmetric w.r.t. the involution $X\mapsto X^\ast:=\{i\in[n]\,\colon\, i^\circ\notin X\}$ for $X\subseteq [n]$, called $\ast$-\emph{symmetric} cubillages), but will play a less important role in our considerations. Note that no skew-symmetric cubillage exists when $n$ is odd. 


\section{Geometric involution} \label{sec4}

The previous definition of involution was combinatorial and quite non-geometric. In this section we give a more enlightening geometric setting.

Let us return to the definition of a zonotope from Section~\ref{sec2}. It dealt with a system of guiding vectors $\xi _1,\ldots,\xi _n$ in $\mathbb R^{[d]}$, and the zonotope $Z(n,d)$ was defined to be the Minkowski sum of $n$ line segments $[0,\xi _i]$, $i=1,\ldots, n$.  From now on it will be more convenient for us to define the zonotope $Z=Z(n,d)$ as the sum of shifted segments $[-\xi _i/2, \xi _i/2]$. Then the symmetry center of  this shifted zonotope is located at the coordinate origin of  $ \mathbb R ^{[d]}$.

Suppose that we have a linear automorphism $A$ of the space $ \mathbb R^{[d]}$ which transfers the vector system $ \xi $ into itself up to signs, that is, $A(\xi _i)$ is $\xi _j$ or $-\xi_j$ for some $j$. In this case, the automorphism $A$ sends the zonotope $Z$ into itself as well. This implies that  $A$ preserves the set of cubillages, sending each $\mathcal Q $ to $A\mathcal Q $.

We use these observations as follows. Let us assign the guiding vectors $ \xi _i$ in a symmetric way on the Veronese curve, namely, $\xi_i=(1, t_i, t_i^2,\ldots, t_i^{d-1})$, and  $t_i=-t_{i^\circ }$. Take the map $A$ that sends the base vector $e_i$ to $(-1)^{d+i-1}e_i$, $i=1,\ldots,d$.
This alternately changes the signs of the base vectors. Also the last base vector $e_d$ is mapped to $-e_d$, and the first base vector $e_1$ to $(-1)^d e_1$. In its turn, $i$-th guiding vector $\xi _i=(1, t_i, t_i^2,\ldots, t_i^{d-1})$ is mapped to $(-1)^d (1,-t_i, t_i^2,\ldots,(-1)^{d-1}t_i ^{d-1})$, that is, to the vector $(-1)^d \xi _{i^\circ }$. Therefore, the symmetric zonotope $Z=Z(n,d)$ is mapped to itself, and the set of cubillages preserves.\medskip

\begin{lemma}\label{lemma1}
For any cubillage $\mathcal Q $ of  $Z (n,d)$, the cubillage $\mathcal Q ^\circ $ is equal to $A\mathcal Q $. 
\end{lemma}

\begin{proof}
Consider the linear automorphism $\widehat{A}$ of the space $\mathbb{R}^{[d+1]}$, which coincides with $A$ on the base vectors $e_i$ for $i\in [d]$, and sends $e_{d+1}$ into itself. Let $Z'=Z([n], d+1)$; then $\widehat{A}$ transfers $Z'$ into itself.

Represent the cubillage $\mathcal Q $ as a membrane $\mathcal M $ in the zonotope $Z'$ (so that $\pi (\mathcal M)=\mathcal Q)$, where $\pi$ is the projection $(x_1,\ldots,x_{d+1})\mapsto(x_1,\ldots,x_d)$) and then embed $\mathcal M$ into  some cubillage $\mathcal Q '$ of the zonotope $Z'$. Then $Inv(\mathcal Q )$ is the set of types of cubes of $\mathcal Q '$ which are located before (relative to the direction of $e_{d+1}$) the membrane $\mathcal M $. The cubillage  $A\mathcal Q $ is obtained as a projection (by $\pi$) of the membrane $\widehat{A} \mathcal M$; therefore, $Inv({A} \mathcal Q)$ is equal to the set of types of cubes of $\mathcal Q'$ located before the membrane $\widehat{A} \mathcal M$. But these cubes are exactly those of the form $\widehat{A}C'$ for $C'$ running over the cubes located before $\mathcal M $. In fact, the direction $e_{d+1}$ preserves under the action of $\widehat{A}$. Therefore, if a cube $C'$ is located before the membrane $\mathcal M $, then the cube $\widehat{A}C' $ is located before $\widehat{A} \mathcal M$.

On the other hand, the type of a cube $\widehat{A}C '$ is equal to the type of $(C')^\circ$. Therefore, $Inv(A \mathcal Q)$ is nothing but $Inv(\mathcal Q^\circ)$.
\end{proof}

This lemma clarifies the geometric meaning of the involution that we consider. To get the cubillage $\mathcal Q ^\circ $, one needs to act on the cubillage $\mathcal Q$ by the diagonal transformation $A$. Since $A$ is an involution, it is given by two eigenspaces with eigenvalues 1 and -1. The first one may be called the {\em axial space} of the reflection $A$. This space is generated by the base vectors $e_{d-1}, e_{d-3},\ldots$\,. The other (``orthogonal'') subspace (in which multiplication by $-1$ acts) is generated by the vectors $e_d, e_{d-2},\ldots$\,.

For example, when $d=1$, we simply ``turn over'' the coordinate. A cubillage corresponds to a permutation of colors; to get the symmetric cubillage, one should reverse the permutation.

When $d=2$, we usually assume that the first coordinate vector is directed  vertically (upward). It forms the axial space of our involution. Then the symmetric tilings are just those that are symmetric with respect to this vertical axis (see Example~2 in Section~\ref{sec3}). 
\medskip

\noindent\textbf{Remark 2.} Similar reasonings show that the skew-symmetric cubillages are those that are invariant with respect to the map $-A$. Note that as is explained in~\cite[Sect.~9]{DKK3}, when both $n,d$ are even, there is a bijection $\Qscr\mapsto \Qscr'$ between symmetric and $\ast$-symmetric cubillages in $Z=Z(n,d)$ (the latter was defined in Remark~1). In particular, this bijection is described there explicitly via a  transformation of the set of vertices. When $d=2$, there is a nice visualization: $\Qscr'$ is obtained from $\Qscr$ by the reflection w.r.t. the SW-to-NE line (at angle of $45^\circ$) through the center of $Z$ (which transforms the vertical axis into the horizontal one).

           
\section{Symmetric flips}\label{sec5}

Next we describe symmetric rearrangements, or \emph{flips}, in symmetric cubillages; they will endow the set ${\bf SQ} (n, d)$ with the structure of a digraph. We start with symmetric cubillages of ``small'' zonotopes, namely,  capsids $Z (d+1, d)$ and barrels $Z (d+2,d)$.

As mentioned in Section~\ref{sec2}, any capsid has exactly two cubillages --  standard and anti-standard ones; each of them is symmetric. Replacing the standard filling in a capsid by the anti-standard one is called the raising flip in this capsid.

A barrel has exactly two symmetric cubillages as well, the standard and anti-standard ones. This can be seen from the fact that Ziegler's trellis for the barrel $B=Z (d+2, d)$ forms one stick with $d+2$ vertices. By the symmetry and bi-convexity, the inversion set contained in this stick is either empty (which corresponds to the standard cubillage), or covers the entire stick (yielding the anti-standard cubillage). This implies that if a symmetric cubillage $\mathcal Q $ has a symmetric barrel fragment, then we can replace its filling by the ``opposite'' one, obtaining a new symmetric cubillage $\mathcal Q '$. We call this transformation a \emph{barrel flip}. (Note that in case $d=2$, a flip of this sort was originally described in~\cite[Sect.~6]{eln}, where it is called an \emph{octagon-flip}.) When needed, this ``complex'' flip can be represented as a composition of $d+2$ usual capsid flips within the barrel, but the capsids involved need not be symmetric, giving non-symmetric intermediate cubillages.

Therefore, symmetric cubillages (elements of ${\bf SQ}(n, d))$ can be connected by three types of ``arrows'' (representing symmetric flips), which we call \emph{simple}, \emph{double}, and \emph{thick}, and denote as $\to $, $\Rightarrow $, and $\Rrightarrow$, respectively. Here a simple arrow appears when there is a symmetric capsid with the standard filling (and therefore one can make a raising symmetric flip). A double arrow appears when there is a pair of symmetric capsids $K$ and $K^\circ$ which have no common cube and both are filled standardly. (It should be noted that capsids $K$ and $K^ \circ$ in a symmetric cubillage are always filled in the same way, either standardly or anti-standardly.) And a thick arrow appears when there is a (symmetric) barrel with the standard filling.

Thus, the set ${\bf SQ}(n, d)$ is endowed with a digraph structure (having  three sorts of edges-arrows). This digraph is acyclic; moreover, it is ranked, where, as before, the rank of a cubillage $\mathcal Q $ is meant to be the size of the inversion set $Inv(\mathcal Q )$. The standard  cubillage is a source (viz. zero-indegree vertex) of this digraph. The important question is: whether or not this is the \emph{unique} source in the digraph? Equivalently, is it true or not that any symmetric cubillage can be reached by a sequence of symmetric raising flips, starting from the standard cubillage? A similar question can be asked about the anti-standard cubillage and sinks (viz. zero-outdegree vertices) in the digraph, but this is reduced to the previous one due to a natural symmetry on ${\bf SQ}(n, d)$. 
\medskip

\noindent\textbf{Conjecture 1.} {\em For any $n,d$, the digraph ${\bf SQ}(n, d)$ has only one source, namely, the standard cubillage of $Z([n],d)$ (and similarly about the anti-standard cubillage).                    }\medskip

A weaker conjecture is that this digraph is connected in the undirected sense (when the directions of edges are discarded); this would immediately follow from validity of Conjecture 1.\medskip

In the next section we verify Conjecture 1 in three special cases: when $d$ is odd and $n$ is even, when $d=2$, and when $n=d+3$.
           
           
\section{Verification of  Conjecture 1 in three special cases} \label{sec6}

\begin{prop}\label{pro1}
Conjecture 1 is valid if $n$ is even and $d$ is odd.
\end{prop}
 
 \Xcomment{
\medskip
  }

\begin{proof}
 Let $ \mathcal Q $ be a symmetric cubillage of the zonotope $Z(n,d)$ different from the standard one. Then (see~\cite[Sect.~12]{UMN-19}) it contains a capsid $K$ with the anti-standard filling. If $K$ is symmetric ($K=K^ \circ $), we can make a simple lowering flip in $\mathcal Q$ and approach the standard cubillage (reducing the system $Inv (\mathcal Q)$  by one element).

Now consider the case when the capsid $K$ is not symmetric,  $K\ne K^\circ $. We claim that these two capsids are innerly disjoint, that is, they have no common interior point (equivalently, no cube in common). Indeed, let $T$ be the type of  $K$. Then the type of $K^\circ $ is equal to $T^\circ $. Suppose that these capsids have a common cube $C$ of type $R$. Then $T=R\cup \{i\}$, $T^\circ =R\cup \{j\}$, and $R=T\cap T^\circ $. It follows that $R =R^\circ $ and $j=i^\circ $. Since $n$ is even, any symmetric subset in $[n]$ has an even size; so is $R$. But $|R|=d$ is odd. A contradiction.

Thus, the capsids $K$ and $K^ \circ $ have no cube in common. Then we can make the lowering flips in both, getting closer to the standard cubillage again. 
  \end{proof}

As a bi-product of the proof, we obtain that when $n$ is even and $d$ is odd, all arrows in the digraph on ${\bf SQ}(n, d)$ are either simple or double, not thick.\medskip

\begin{prop}\label{pro2}
Conjectire 1 is valid if $n=d+3$. 
\end{prop}

\begin{proof}
In view of Proposition~\ref{pro1}, we may assume that $d$ is even.  Consider the corresponding Ziegler's trellis. It has a rather simple structure, illustrated in Fig.~\ref{fig6}.

\begin{figure}[htb]
\begin{center}
\includegraphics[scale=0.8]{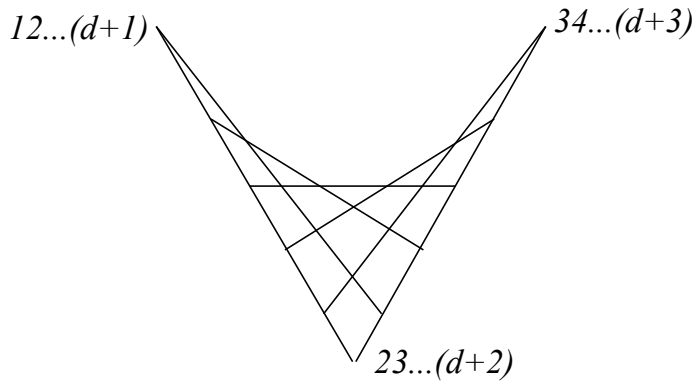}
\end{center}
\vspace{-0.4cm}
\caption{Trellis $Tre (d+3, d)$ for $d$ even.}
 \label{fig6}
\end{figure}

In this trellis, there are three corner vertices and they correspond to the interval subsets of size $d+1$ in $[n]$. Also there are $n=d+3$ sticks, each containing $d+2$ vertices, and each vertex belongs to two sticks. A feature of the situation is the presence of one ``horizontal'' stick-line $H$ (related to the packet $2\cdots (d+2)$). The remaining sticks are partitioned into two sets, one consisting of \emph{NW--SE} sticks, and the other of \emph{SW--NE} sticks, which are symmetric to each other. Note that each vertex lying above  $H$ belongs to two sticks of the same set, while each vertex below $H$ belongs to two sticks from different sets.

Let $\mathcal Q $ be a symmetric cubillage different from the standard one. Then its inversion system $S=Inv (\mathcal Q )$ is nonempty, symmetric and bi-convex. 
Let $v$ be a vertex of $S$ with the maximum height in the trellis. Consider four cases.
\smallskip

1. Suppose that $v$ lies above $H$ and is not a corner. Let $A,B$ be the sticks containing $v$; w.l.o.g., one may assume that they are NW--SE sticks. Then their symmetric sticks $A',B'$ in the trellis are SW--NE sticks, and their common vertex $v'$ is different from $v$ and belongs to $S$. Considering the intersections of $A$ and $B$ with $A'$ and $B'$, and using the bi-convexity and the fact that $v$ is maximal but not a corner, one can realize that all vertices of $A$ and $B$ below $v$ belong to $S$, and therefore, the removal of $v$ from $S$ preserves the bi-convexity. And similarly for the removal of $v'$. Moreover, the removal of both $v,v'$ preserves the symmetry. This means that we can make a double lowering flip in $\mathcal Q $.
\smallskip

2. When $v$ lies below $H$, we argue in a similar way with the only difference that $v$ may be self-symmetric, in which case we deal with a single flip in $\mathcal Q $.
\smallskip

3. Now let $v$ lie on $H$. By the symmetry and bi-convexity, 
all vertices of $H$ must belong to  $S$. This implies that all vertices lying below $H$ must belong to $S$ as well (since they lie on slanted sticks meeting $H$). Therefore, the removal of all vertices of $H$ from $S$ preserves both the bi-convexity and symmetry. In view of $n=d+3$, this removal corresponds to a barrel flip in $\mathcal Q $.
 \smallskip
 
4. Finally, let $v$ be an upper corner. Then the other upper corner $v'$ is in $S$ as well. Let $\Ascr$ be the set of NW--SE sticks plus, possibly, $H$ whose intersections with $S$ give initial intervals in the sticks; this $\Ascr$ is nonempty as it includes the leftmost stick (containing $v$). Let $\Ascr'$ be symmetric to $\Ascr$. Now consider the set $X$ of vertices of $S$ with the minimal height among those contained in $\Ascr\cup\Ascr'$.  If $X$ covers $H$, we act as in case~3. Otherwise choose in $X$ an arbitrary vertex $w$ and its symmetric vertex $w'$ (possibly $w=w'$). One can check that the removal of such $w,w'$ preserves the bi-convexity, yielding a single or double lowering flip.
 \end{proof}

 \begin{prop}\label{pro3} Conjecture 1 is valid if $d=2$. 
\end{prop}

\noindent(For $n$ even, a weaker version, concerning merely the connectedness of ${\bf SQ}(n, 2)$, was shown in~\cite[Prop.~6.4]{eln} (by use of a general result on reduced words in Coxeter groups) and in~\cite[Theorem~3.4]{DKK3} (directly); in both cases, the description is given in terms of $\ast$-symmetric tilings, defined in Remark~2.)
\medskip

\begin{proof}
Let $m:=\lfloor n/2\rfloor$. Consider a symmetric rhombus tiling $\mathcal T $ of the zonogon $Z=Z(n,2)$ (where the symmetry is taken with respect to the vertical axis, such as in Fig.~\ref{fig1} or~\ref{fig4}). A feature of such a tiling is the presence of a vertical \emph{spine} consisting of $m$ symmetric rhombuses cut by the axis (plus one edge of the middle color $m+1$ if $n=2m+1$).

We assert that by making a sequence of  symmetric lowering flips (simple, double, and/or barrel ones) one can transform $\mathcal T $ into the standard tiling. 

To show this, assume that $\mathcal T $ is not standard. Then it has a capsid (hexagon) $H$ with the anti-standard filling (consisting of three rhombi of which two share the topmost vertex of the hexagon). (a) If $H$ is symmetric, then one can make a simple lowering flip in it. (b) If $H$ is innerly disjoint from the symmetric hexagon $H^\circ $, then one can make a double lowering flip in this pair of hexagons. In both cases, we can use induction.

A more difficult situation (which is principal for the Conjecture 1 with $d$ even in general) arises  when these hexagons are different but have a common interior point. One can see that in this case $H\cap H^\circ$ consists of one symmetric rhombus. This situation is illustrated in Fig.~\ref{fig7}.

\begin{figure}[htb]
\begin{center}
\includegraphics[scale=0.8]{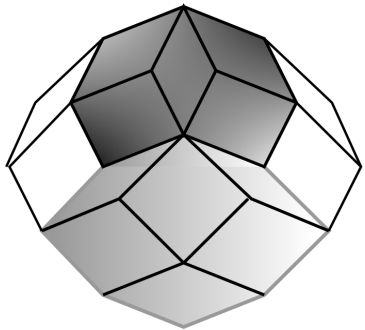}
\end{center}
\vspace{-0.4cm}
\caption{No barrel at the top.}
  \label{fig7}
\end{figure}

\noindent(This picture illustrates a symmetric tiling of the zonogon $Z(6,2)$. Here in the upper part, two (symmetric to each other) hexagons, shadowed darker, share a symmetric rhombus. These two hexagons are not extended to a  barrel (octagon) in the tiling; so we cannot make a barrel flip at this place. Fortunately, there are two other hexagons below (shadowed lighter), which have the anti-standard fillings and already admit an extension to a barrel. So we can make a lowering flip in this barrel.)

The above trouble is overcome as follows. Choose a hexagon $H$ in $ \mathcal T$ so that it has the anti-standard filling and is located as low as possible. We may assume that we are not in cases (a),(b) above. Then the hexagons $H$ and $H^\circ$ (filled anti-standardly as well) share one rhombus $R$. Let $v$ be the lowest vertex of $R$, and let $e$ ($e'$) be the edge in $H$ (resp. $H^\circ$) entering $v$. Two cases are possible: (i) $e,e'$ are neighboring edges in $ \mathcal T$ entering $v$; and (ii) there is an edge $e''$ in  $\mathcal T$ entering $v$ and lying between $e$ and $e'$. 

In case~(i), the edges $e,e'$ belong to one rhombus in $ \mathcal T$, and adding it to $H$ and $H^\circ$, we obtain a barrel with the anti-standard filling, yielding the result.

And in case~(ii) (illustrated in Fig.~\ref{fig8}), one can show (see e.g.~\cite[Lemma~2.6]{UMN-10}) the existence of one more hexagon in $ \mathcal T$ which is filled anti-standardly and lies lower than $H$, yielding a contradiction.
 \end{proof}

 \begin{figure}[htb]
\begin{center}
\includegraphics[scale=.25]{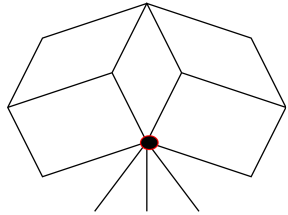}
\end{center}
\vspace{-0.4cm}
\caption{``Chandelier'' (an illustration to case~(ii) in the proof of Proposition~\ref{pro3}).}
 \label{fig8}
\end{figure}


\section{Higher Bruhat of types B and C}\label{sec7}

\noindent\textbf{Definition.}
Recall that for positive integers $n$ and $d\le n$, depending on the context, we may use notation $\mathbf{SQ}(n,d)$ either for the set of symmetric cubillages of the zonotope $Z(n,d)$, or for the corresponding acyclic digraph (when the edges-arrows associated with raising symmetric flips are included in consideration), or for the poset obtained as the transitive closure of the above digraph; see Sections~\ref{sec2} and~\ref{sec5}. 
In the second (third) case, when $n$ is even, $n=2m$, we refer to the digraph (resp. poset) $\mathbf{SQ}(n,d)$ as  $d$-th \emph{higher Bruhat digraph ({\rm resp.} order) of type $C_m$} and denote it as $\mathbf{C} (m,d)$. When $n$ is odd, $n=2m+1$, the digraph (poset) $\mathbf{SQ}(n,d)$ is referred to as $d$-th \emph{higher Bruhat poset ({\rm resp.} order) of type $B_m$}, denoted  as $\mathbf{B} (m,d)$.
 \medskip

\noindent\textbf{Example 4.} 
Return to Example 1 from Section~\ref{sec3}. For  $d=1$, the  Bruhat of types $C$ and $B$ coincide as sets.  Moreover, they coincide as digraphs as well. The only difference is that a simple arrow in $\mathbf{C} (m, 1)$ becomes a thick arrow in $\mathbf{B}(m,1)$.

Figure~\ref{fig9} illustrates the trellises $Tre(4,1)$ and $Tre(5,1)$, and below we describe the digraphs ${\bf SQ} (4,1)$ and ${\bf SQ} (5,1)$, viz. the first Bruhat of types $C_2$ and $B_2$. The symmetric bi-convex subsets occurring in the former trellis are: 

 \begin{gather*}
O=\emptyset, \ A=\{24\}, \ B=A\cup \{14, 25\}, \ C=B\cup \{15\}, \ D=C\cup \{12, 45\}=\overline{O}, \\
\mbox{and} \ \overline{A},\, \overline{B},\,\overline{C} \ \mbox{(the complements to} \ A, \, B,\, C),
  \end{gather*}

\noindent considering the color set $\{1,2,4,5\}$. In this case, the digraph ${\bf SQ} (4,1)$ is viewed as

$$
\begin{array}{ccccccc}
\overline{C} & \rightarrow & \overline{B} & \Rightarrow & \overline{A} & \rightarrow & \overline{O} \\
  \Uparrow &  &  &  &  &   & \Uparrow \\
O & \rightarrow & A & \Rightarrow & B & \rightarrow & C 
\end{array}
$$

The digraph ${\bf SQ} (5,1)$ (with the color set $\{1,2,3,4,5\}$) is similar, up to replacing the simple arrows by thick (barrel)  ones, and regarding $A$ as $\{23,24,34\}$, and $C$ as $B\cup\{13,15,35\}$.
 \medskip

 \begin{figure}[htb]
\begin{center}
\includegraphics[scale=.5]{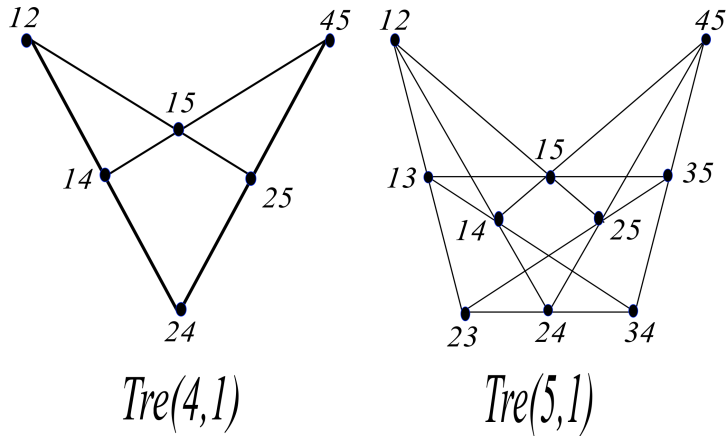}
\end{center}
\vspace{-0.4cm}
\caption{Trellises for ${\bf Q}(4,1)$ and ${\bf Q}(5,1)$.}
 \label{fig9}
\end{figure}

\vspace{-0.2cm}
 \begin{figure}[htb]
\begin{center}
\includegraphics[scale=.25]{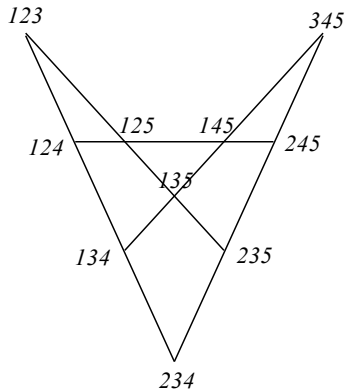}
\end{center}
\vspace{-0.4cm}
\caption{Trellis for ${\bf Q}(5,2)$.}
 \label{fig10}
\end{figure}

\noindent\textbf{Example 5.} Now consider the second Bruhat digraphs of types $C_2$ and $B_2$, that is, ${\bf SQ}(4,2)$ and ${\bf SQ} (5,2)$. The former digraph is pretty simple: it is formed by one thick arrow going from the standard tiling to the anti-standard one. The latter digraph is less trivial. The corresponding Ziegler's trellis is illustrated in Fig.~\ref{fig10}.
Here the  inversion system consists of the following sets (in the increasing order): 
 $$
O=\emptyset, \ A=\{234\}, \ B=A\cup \{134,235\}, \ C=B\cup \{135\}, \ D=C\cup \{124,125,145,245\},
 $$

\noindent and their complements. Then the digraph {\bf SQ} (5,2) is viewed as follows:

$$
\begin{array}{ccccccccc}
  \overline{D} & \Rrightarrow & \overline{C} & \to & \overline{B} & \Rightarrow & \overline{A} & \to & \overline{O} \\
  \Uparrow &  &  &  &  &  &  &  & \Uparrow \\
  O & \to & A & \Rightarrow & B & \to & C & \Rrightarrow & D
\end{array}
$$

\noindent\textbf{Example 6.} For reader's curiosity, we illustrate the digraph ${\bf SQ}(6,3)$ for the symmetric 3-dimensional cubillages with 6 colors, that is, $\textbf{C}(3,3)$. Here, to simplify the picture, we do not distinguish between simple and double arrows.

{\scriptsize
 $$
\begin{array}{ccccccccc}
   &  &  &  & \bullet &  &  &  &  \\
   &  &  & \nearrow &  & \nwarrow &  &  &  \\
   &  & \bullet &  &  &  & \bullet &  &  \\
   &  & \uparrow &  &  &  & \uparrow &  &  \\
   &  & \bullet &  &  &  & \bullet &  &  \\
   & \nearrow &  & \nwarrow &  &  & \uparrow &  &  \\
  \bullet &  &  &  & \bullet &  & \bullet &  &  \\
   & \nwarrow &  & \nearrow &  &  & \uparrow &  &  \\
   &  & \bullet &  &  &  & \bullet &  &  \\
   &  & \uparrow &  &  &  & \uparrow &  &  \\
   &  & \bullet &  &  &  & \bullet &  &  \\
   &  & \uparrow &  &  & \nearrow &  & \nwarrow &  \\
   &  & \bullet &  & \bullet &  &  &  & \bullet \\
   &  & \uparrow &  &  & \nwarrow &  & \nearrow &  \\
   &  & \bullet &  &  &  & \bullet &  &  \\
   &  & \uparrow &  &  &  & \uparrow &  &  \\
   &  & \bullet &  &  &  & \bullet &  &  \\
   &  &  & \nwarrow &  & \nearrow &  &  &  \\
   &  &  &  & \bullet &  &  &  &
\end{array}
$$
}


                        \section{Comparison of B and C}\label{sec8}

The digraphs ${\bf SQ}(2m+1, d)$ and ${\bf SQ}(2m, d)$ (viz. the higher Bruhat digraphs $\textbf{B} (m, d)$ and $\textbf{C} (m,d)$) can be connected by use of a natural mapping of ``reduction of the middle color''. Namely, for a cubillage $\mathcal Q $ of the zonotope $Z(2m+1, d)$, one can reduce (withdraw) the middle color $m+1$ (for details see \cite[Sect.~4]{UMN-19}), which gives a cubillage of the zonotope $Z_{red}\simeq Z(2m, d)$, denoted as $\mathcal Q _{red}$ or $\mathcal Q _0$. On the combinatorial language of inversion systems, this is viewed as follows.
Using a symmetric notation for colors, we write
      \begin{eqnarray}
         \langle 2m \rangle &\mbox{for}& \{-m,\ldots,-1,1,\ldots,m\}, \quad\mbox{and} \nonumber \\
         \langle 2m+1 \rangle &\mbox{for} & \{-m,\ldots,-1,0,1,\ldots,m\}=\langle 2m \rangle\cup \{0\}. \nonumber
          \end{eqnarray}
Then $i^ \circ =-i$, and the middle color 0 is symmetric to itself.

Recall that in a combinatorial setting, the cubillages of the zonotope $Z (n,d)$ are represented as bi-convex subsets of vertices in the (abstract) trellis $Tre ([n], d)$ with the set of vertices $Gr([n], d+1)$ and the set of sticks (hyper-edges)  $Gr ([n], d+2)$. Symmetric cubillages of $Z(2m+1,d)$ form symmetric vertex subsets in $Tre(\langle 2m+1 \rangle,d)$. In its turn, the trellis $Tre(\langle 2m \rangle, d)$ can be naturally realized as a sub-trellis of $Tre (\langle 2m+1 \rangle, d)$ agreeable with the color involution $\circ $. 
See Fig.~\ref{fig9} for an example.

Once $Inv(\mathcal Q )$ is the set of inversions of a symmetric cubillage $\mathcal Q$ of the zonotope $Z(\langle 2m+1 \rangle, d)$, the intersection of $Inv (\mathcal Q )$ with the sub-trellis $Tre (\langle 2m \rangle, d)$ is symmetric and bi-convex, which implies that the set $Inv(\mathcal Q _0)$ of inversions of the reduced cubillage $\mathcal Q _0$ is equal to
                                       $$
                            Inv(\mathcal  Q )\cap Gr(\langle 2m \rangle,d).
                                                 $$
Then the reduction of the middle color 0 gives the mapping
                                                   $$
     red: {\bf SQ}(\langle 2m+1 \rangle, d) \to  {\bf SQ}(\langle 2m \rangle, d).
                                                     $$

\noindent Note that this is not simply a mapping of sets, as in fact, $red$ is consistent with the structure of digraphs, in the sense that each arrow is mapped to an arrow\footnote{Our definition here is slightly different from the standard definition of a digraph morphism; to get a complete correspondence, we should add ``loops'' (edges connecting identical vertices).}.

To illustrate this construction, we exhibit in Fig.~\ref{fig11} the instance $red: {\bf SQ}(\langle 5\rangle, 2) \to {\bf SQ}(\langle 4\rangle, 2)$, using notation as in Example 5 from the previous section.

\[
\begin{array}{ccc}
  \overline{D} & \Rrightarrow & \overline{C} \\
  \Uparrow &  & \downarrow \\
  O &  & \overline{B} \\
  \downarrow &  & \Downarrow \\
  A &  & \overline{A} \\
  \Downarrow &  & \Downarrow \\
  B &  & \overline{O} \\
  \downarrow &  & \Uparrow \\
  C & \Rrightarrow & D \\
   & \downarrow & red \\
  \circ & \Rrightarrow & \bullet
\end{array}
\]

\vspace{-1.2cm}

 \begin{figure}[htb]
\begin{center}
\end{center}
\caption{The mapping $red: {\bf SQ}(\langle 5\rangle, 2) \to {\bf SQ}(\langle 4\rangle, 2)$.}
 \label{fig11}
\end{figure}

\begin{prop}\label{pro4}
 The mapping $red$ is consistent with the structure of  digraphs.
 \end{prop}

\begin{proof}
Let $\mathcal Q$ be a cubillage in ${\bf SQ}(\langle 2m+1 \rangle, d)$ and $S:=Inv(\mathcal Q)$. Let $\mathcal Q '$ be obtained from $\mathcal Q $ by a raising flip. Denote the inversion set of $\mathcal Q'$ by $S'$. Flips in $\mathcal Q$ can be of three sorts: simple, double, and barrel ones. Consider these cases.\medskip

\emph{Simple flip.} It deals with some standard capsid which is symmetric; let $K$ be its type. Then $K=K^\circ $ and $Inv(\mathcal Q')=Inv(\mathcal Q )\cup \{K\}$; accordingly, the ``vertex'' $K$ is added in the trellis $Tre (\langle 2m+1 \rangle, d))$. We claim that $0\in K$. For otherwise, we can form the symmetric packet $F=K\cup \{0\}$ and the corresponding symmetric stick $Gr(F, d+1)$ containing the vertex $K$ in the center. The intersection of this stick with the set $Inv(\mathcal Q )$ (which is symmetric and bi-convex) is empty. After adding the vertex $K$, it must remain bi-convex, which is impossible.

So $0\in K$, implying that the size of $K$ (equal to $d+1$) is odd, whence $d$ is even. Then $K$ (as a vertex of $Tre(\langle 2m+1 \rangle, d)$) does not belong to $Tre(\langle 2m \rangle, d)$, and adding $K$ does not change the intersection with $Tre (\langle 2m \rangle, d)$. So $red(\mathcal Q ')=red(\mathcal Q )$.\medskip

\emph{ Double flip.} It is obtained by adding to $Inv(\mathcal Q )$ two (symmetric to each other) vertices $v$ and $v'$ in the trellis $Tre (\langle 2m+1 \rangle, d)$. If they do not belong to $Tre (\langle 2m \rangle, d)$, then nothing is changed under the reduction. And if they (simultaneously) belong to $Tre (\langle 2m \rangle, d)$,  
then the double flip is applicable after the reduction.\medskip

\emph{Barrel flip.} It is obtained by adding to $Inv(\mathcal Q )$  all vertices of a new (barrel) stick $R$. Moreover, since the barrel is symmetric, this stick should also be symmetric and, therefore, ``horizontal''. Two cases are possible, depending on the parity of $d$. If $d$ is even, then 0 is not 
included in the barrel type, so all vertices of  $R$ (as elements of $Gr(\langle 2m+1 \rangle, d)$) are  contained in $\langle 2m \rangle$, that is, $R$ is also a stick in the trellis $Tre (\langle 2m \rangle, d)$. Under the reduction, this stick is added to $Inv(\mathcal Q _{0})$, and we get a barrel flip as well.

And if $d$ is odd, then the color 0 is included in the barrel type. All vertices in the stick $R$ contain the color 0, except for the central vertex. Therefore, the raising barrel flip causes adding this central vertex. This gives a simple flip after the reduction. 
  \end{proof}

 \begin{theorem}\label{theorem1}
The reduction mapping
                                       $$
      red: {\bf SQ}(\langle 2m+1 \rangle,d)\to {\bf SQ}(\langle 2m \rangle,d)
                                         $$
is surjective and all its fibers are connected.
\end{theorem}

 \begin{proof} 
 The reduction mapping was defined combinatorially. To prove the theorem, it is convenient to use a geometric description of the reduction, which we now give.

Let $ \mathcal Q $ be a cubillage of the zonotope $Z=Z(\langle 2m+1 \rangle, d)$ with the middle color 0. Let $P$ be the pie in $\mathcal Q $ associated with  color 0. 

[For a review in detail on pies, see~\cite[Sect.~3]{UMN-19}. Briefly: $P$ consists of all cubes in $\Qscr$ containing color 0 in their types. Each cube $C$ in $P$ is the Minkowsky sum of some facet $F$ and the segment $[0,\xi_0]$; the union  of these facets gives a $(d-1)$-dimensional subcomplex (forming a ``disk'') $\Mscr$ in $\Qscr$, so that $P$ is viewed as the ``direct product'' of $\Mscr$ and $[0,\xi_0]$.  Compressing (or ``reducing'') $P$ in the direction $\xi _0$ (which eliminates color 0 and turns $P$ into $\Mscr$), we obtain the reduced cubillage $ \mathcal Q _0$; we call $\Mscr$ a membrane in $\mathcal Q _0$ (as well as in $\Qscr$) relative to the direction $\xi _0$, or a \emph{0-membrane}. A general 0-membrane $\Mscr$ in $\Qscr$ has the property that under the projection $\pi_0: \Rset^{[d+1]}\to \Rset^{[d]}$ along $\xi _0$, ~$\mathcal M$ is bijectively mapped to the image $\pi _0 (Z)\simeq Z(\langle 2m \rangle, d-1)$ of the entire zonotope $Z$.  Note that, by the Veronese representation of $ \xi _0$ with $t_0=0$ (as in Sect.~\ref{sec4}), $\xi_0$ is exactly the coordinate vector $e_1$, which for convenience is thought of as being directed in the vertical direction (upward). It follows that the zonotope $Z_0$ is cyclic as well.]

Thus, the cubillage $\mathcal Q $ determines the cubillage $\mathcal Q _0$ and a 0-membrane $\mathcal M $ in the latter. Conversely, given a cubillage $\mathcal Q _0$ of the zonotope $Z_0=Z(\langle 2m \rangle, d)$ and a 0-membrane $\mathcal M $ in it, the ``expansion operation'' w.r.t. $\xi _0$ applied to $\Mscr$ (which expands each facet in $\Mscr$ to the corresponding cube by adding color 0) results in a correct cubillage $\mathcal Q $ of  $Z=Z(\langle 2m+1 \rangle, d)$. These operations are inverse to each other. Moreover, if $\mathcal Q $ is symmetric, then so are the cubillage $\mathcal Q _0$ and the 0-membrane $\mathcal M $ in it. Therefore, ${\bf SQ}(\langle 2m+1 \rangle, d)$ can be identified with the set of pairs $(\mathcal Q _0,\mathcal M)$, where $\mathcal Q _0$ runs over ${\bf SQ} (\langle 2m \rangle, d)$, and $\mathcal M $ over the set of symmetric 0-membranes in $\mathcal Q _0$. In other words, the fiber of the reduction mapping $red$ over $ \mathcal Q _0$ is identified with the set of symmetric 0-membranes in $\mathcal Q _0$.

Now the proof of our theorem falls into two tasks:

(i) to show that every symmetric cubillage $\mathcal Q _0$ of ${\bf SQ}(\langle 2m \rangle, d)$ has a symmetric 0-membrane $\mathcal M $;

(ii) to show that any two symmetric 0-membranes in $\mathcal Q _0$ are connected by a sequence of 0-flops within $\mathcal Q _0$ (where, following terminology in~\cite{UMN-19},  a \emph{flop} means the transformation of one membrane into another corresponding to a flip in the related cubillage of the previous dimension). 

In proving the assertions in (i) and (ii), our reasonings depend on the parity of $d$. Let us look at the cubillage $\mathcal Q _0$ in the direction of $e_1$. This makes visible the ``lower'' side of the boundary of the zonotope $Z_0$ (the ``bottom''). If $d$ is even, then the projection along $e_1$ does not affect the visible side in essence, implying that the bottom forms a symmetric 0-membrane. And similarly for the ``upper'' (invisible) part of the boundary of $Z_0$, the ``roof''. This immediately gives the assertion in~(i) when $d$ is even.

Assuming that $d$ is even, the assertion in (ii) is shown as follows. We use the concept of {\em natural 0-order} on the set of cubes of the cubillage $\mathcal Q _0$. It is defined via the ``shading relation'' (in spirit of~\cite{UMN-19}), but now the shading is regarded relative to the vertical direction. [More precisely, for cubes $C,C'$ in $\mathcal Q _0$, we say that $C$  \emph{immediately 0-precedes} $C'$ if they have a common facet lying in the upper boundary of $C$ (and the lower boundary of $C'$).  One shows that this relation is acyclic, and therefore, it determines a partial order $\prec_0$ on the cubes.]  So for any 0-membrane, we have a set (``stack'') of cubes lying ``below'' the membrane. Now let $ \mathcal M $ be a symmetric 0-membrane, and let $C$ be a maximal (relative to $\preceq _0$) cube in the stack of $\mathcal M $. If $C$ is symmetric, then we can make a ``lowering'' flop in $\mathcal M $. And if the cube $C^\circ $ symmetric to $C$ is different from $C$, then it must lie below $ \mathcal M $ as well, and we can simultaneously make lowering flops in $C$ and $C^ \circ $, getting closer to the ``bottommost'' membrane of $\mathcal Q _0$. This proves the assertion in (ii). Moreover, it follows that a minimal 0-membrane is unique.
\smallskip

Now let us assume that $d$ is odd. In this case, the symmetry turns the bottom of $Z_0$ into the top (``roof''), and vice versa. So these two 0-membranes are not (self)-symmetric, but they are symmetric to each other. And we can make a raising flop of the bottom, and symmetrically, a lowering flop of the top, resulting in two closer 0-membranes symmetric to each other. At a general step of this process, we have a 0-membrane $ \mathcal M $ and its symmetric 0-membrane $\mathcal M '=\mathcal M ^\circ $ such that $\mathcal M $ is located lower than $\mathcal M '$. If they coincide, we obtain a symmetric 0-membrane. If they differ, then the gap between $\mathcal M $ and $\mathcal M '$ is filled with some cubes of the cubillage $\mathcal Q _0$. Let $C'$ be a maximal cube among them (relative to $\prec _0$).

We observe that this cube is not symmetric. For otherwise, its type would be a symmetric (relative to $ \circ $) subset of the odd size $d$ in $\langle 2m \rangle$, which is impossible. So the cube $C$ symmetric to $C'$ is different from $C'$, and the lower side of $C$ is contained in the membrane $\mathcal M$. Then we can make a raising flop in $ \mathcal M $ (relative to the cube $C$) and simultaneously a lowering flop in $\mathcal M '$ (relative to $C'$), which gives a closer pair of symmetric 0-membranes, and the assertion in~(i) follows.

The assertion in~(ii) is obtained in a similar fashion. Let $\mathcal M $ and $\mathcal M '$ be two symmetric 0-membranes. Let $sub(\mathcal M )$ be the set of cubes of $\mathcal Q _0$ lying below $\mathcal M $, and similarly for $\mathcal M '$. If $\mathcal M \ne \mathcal M' $, then there is a cube $C$ contained in $sub (\mathcal M )$ but not in $sub(\mathcal M ')$. Assume that $C$ is maximal in $sub (\mathcal M )- sub (\mathcal M')$. Then its upper boundary lies on $ \mathcal M $. This implies that $\mathcal M $ contains the lower boundary of the symmetric cube $C'$ (which lies above $\Mscr$). Also since $C$ lies above $\mathcal M '$, the cube $C'$ lies below $\mathcal M'$.
 \medskip

\noindent\textbf{Claim.} {\em The cubes $C$ and $C'$ have no facet in common.}\medskip

\noindent Indeed, a common facet $F$ would lie on the upper boundary of  $C$ and on the lower boundary of $C'$ (in view the location of $C$ and $C'$ with respect to $\mathcal M $). Therefore, $C\prec _0 C'$. But then $C'\in sub(\mathcal M ')$  implies $C\in sub(\mathcal M ')$; a contradiction. \hfill$\qed$
\medskip

By the Claim,  in the membrane $\mathcal M $, we can make the lowering flop w.r.t. $C$, and simultaneously the raising flop w.r.t. $C'$. This gives a new symmetric 0-membrane which becomes closer to $ \mathcal M '$. 
  \end{proof}

\noindent\textbf{Definition.} We say that a surjective morphism $\Gamma \to \Delta$ of digraphs is \emph{full} if any edge-arrow in $\Delta$ has its preimage edge-arrow in $\Gamma$.\footnote{This resembles, though slightly differs from, the  concept of full functors in category theory.}\medskip
  
\noindent\textbf{Conjecture 2.} {\em The mapping $red: {\bf SQ}(\langle 2m+1\rangle, d) \to {\bf SQ}(\langle 2m\rangle, d)$ is full.}\medskip

It can be shown that this conjecture is valid for $d=2$. (The case $m=d=2$ is illustrated in Fig.~\ref{fig11}.)


\section{Relation to Bruhat of type A}\label{sec9}

There is a nontrivial relationship between higher Bruhat digraphs of types B-C and those of type A. For simplicity, we describe this for the case of even number of colors,  $n=2m$, and start with the case of $d$ even.

Recall that in Section~\ref{sec4} we introduced the notion of axial space in $ \mathbb R^{[d]}$. This is the subspace generated by the base vectors $e_{d-1}, e_{d-3},\ldots$\,; when $d$ is even, there are exactly $d/2$ such vectors. 
We denote the axial space as $ \mathbb R ^{odd}$. 
Consider its intersection with the (symmetric) zonotope $Z=Z(\langle 2m \rangle, d)$. This is a convex symmetric (w.r.t. the origin) polyhedron; let us denote it as $X$ and call the {\em axis} of the zonotope $Z$.

We are interested in the intersection of  a symmetric cubillage $ \mathcal Q $ and the axis $X$. More precisely, we consider a cubillage $\mathcal Q $ as a polyhedral decomposition (subdivision) of $Z$, and, as before, mean by a \emph{face} $F$ of $\mathcal Q $ a face of some cube in it. Consider  the intersection of this decomposition and the axis $X$. It consists of cells of the form $F\cap X$, where $F$ runs over the faces of $ \mathcal Q $; the resulting polyhedral decomposition of $X$ is denoted by $\mathcal Q \cap X$ or $cor(\mathcal{Q})$ and we call  it the \emph{core} of the cubillage $\mathcal Q$. The main statement of this section is the following

\begin{theorem}\label{theorem2}
For $n,d$ even, the axis $X $ of $Z=Z(n,d)$  is a cyclic zonotope, and for a cubillage $\mathcal Q$ of $Z$,  $cor(\mathcal{Q})$ is a  cubillage of $X$.  
  \end{theorem}
  
\begin{proof}
The main role in the construction of the complex $cor(\mathcal{Q})=\mathcal Q \cap X$ is played by the symmetric faces of the complex $\mathcal Q $. This is specified as follows. Let $x$ be a point of the axis $X$, and denote by $F(x)$ the minimal face of $ \mathcal Q $ containing $x$. Then $F(x)$ is a symmetric face, that is, $AF(x)=F (x)$, where $A$ is the automorphism defined in Section \ref{sec4}. Indeed, the symmetric face $AF (x)$ also contains $x$, and therefore it coincides with $F (x)$.

This motivates the following question: how does the axis $X$ pass across a symmetric face $F$ of  $ \mathcal Q $ ? We may assume that, up to  an appropriate shifting, the center of $F$ is situated at the origin of coordinates. The face $F$ is a cube of dimension $\le d$ with guiding vectors $ \xi _i$ whose colors $i$ form a symmetric subset $R$ of $\langle 2m \rangle$; in particular, the dimension of $F$ is even. The involution $A$ transfers $\xi _i$ into $\xi _{-i}$. This easily implies that 
$F\cap X$ is a cube with the guiding vectors $\eta _i=\xi _i+\xi _{-i}$; so the dimension of $F\cap X$ is twice smaller than the dimension of $F$.

Thus, we obtain that $X$ is a zonotope with the guiding vectors $\eta _i$, $i=1,\ldots, d/2$, and that the intersection $ \mathcal Q \cap X$ is a cubillage of the axis $X$.  Each $d/2$-dimensional cell of this cubillage is of the form $C\cap X$, where $C$ is a symmetric $d$-dimensional cube of the cubillage $\mathcal Q $.  Assuming that the guiding vectors $ \xi _i$ are given by a symmetric Veronese formula (see Section~\ref{sec4}), the guiding vectors in $X$ are viewed as $\eta _i=(2,0, 2t_i^2,0,\ldots, 2t_i^d, 0)$. Then each $\eta _i$ is the doubled vector $(1, s_i,\ldots, s_i^{d/2})$, where $s_i=t_i^2$. This implies the cyclicity of $X$.
  \end{proof}

\textbf{Example 7.} Figure~\ref{fig12} shows the standard rhombus tiling of the zonogon $Z(6,2)$ and its vertical axis. The right fragment illustrates the same axis with the induced core, which is a 1-dimensional cubillage with the sequence 1,2,3 of colors. In this example, the core corresponds to the standard (natural) linear order 123. This is not simply a coincidence, as is seen from the next proposition.

 \begin{figure}[htb]
\begin{center}
\includegraphics[scale=.4]{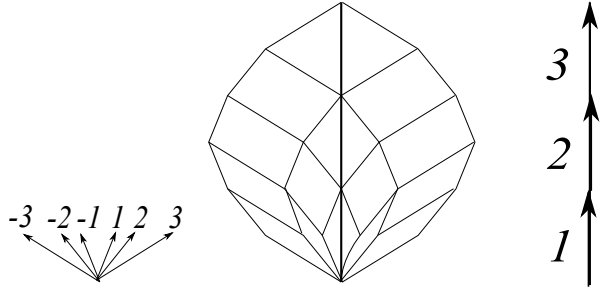}
\end{center}
\caption{The standard tiling of the zonogon $Z (6,2)$ and its core (on the right).}
  \label{fig12}
\end{figure}

\begin{prop}\label{pro6}
If $\mathcal Q $ is the standard cubillage of the zonotope $Z(\langle 2m \rangle, d)$ ($d$ is even), then $cor(\mathcal Q )$ is the standard cubillage of $Z ([m], d/2)$.\end{prop}

 \begin{proof}
The easiest way to show this is to examine spectra (vertex sets) of cubillages. One can see that the vertices of the cubillage $cor (\mathcal Q )$ (for any symmetric $ \mathcal Q $, not only for the standard one) are exactly the vertices of $\mathcal Q $ lying on the axis. 

It is known (see e.g.~\cite{UMN-19}) that any cubillage is determined by its spectrum, and that the spectrum of the standard cubillage of  $Z (\langle 2m \rangle,d)$ consists of the subsets of $\langle 2m\rangle$ that are representable as the union of $d$ intervals, $I_1\cup I_2\cup \cdots\cup I_d$. On the other hand, if a vertex of a symmetric cubillage lies on the axis, then it is invariant with respect to the involution $i \mapsto i^\circ =-i$. Therefore, $I_1=-I_d$, $I_2= - I_{d-1}$, and so on. Such a vertex, regarded as a point in the zonotope $Z ([m], d/2)$, has the spectrum consisting of $d/2$ intervals in $[m]$, namely, $I_{d/2+1},\ldots, I_d$. So we get exactly the spectrum of the standard cubillage $Z ([m], d/2)$.
  \end{proof}
  
Theorem \ref{theorem2} shows that for $d$ even there is a mapping
                        $$
                           cor: {\bf SQ}(\langle 2m \rangle,d) \to {\bf Q}(m,d/2),
                          $$
which sends a symmetric cubillage $\mathcal Q $ into the cubillage $\mathcal Q \cap X$ of the axis $X$. However, this is not simply a mapping of set-systems, but this is consistent with the structure of digraphs. Indeed, consider a raising flip in a symmetric cubillage $\mathcal Q $. There are no simple flips when $d$ is even, since there are no symmetric capsids in this case. A double flip in $\mathcal Q $ does not use symmetric cubes of  $\mathcal Q $, so it does not affect the  intersection with $X$ either. It turn, a barrel flip does lead to a simple flip in $\mathcal Q \cap X$, which is a raising one, as can be concluded from Proposition \ref{pro6}.

Thus, we obtain a mapping of the higher Bruhat digraphs (orders)
             $$
                               {\bf C}(m,d) \to  {\bf A}(m,d/2).
               $$
One can show that a similar machinery works (without big changes) in the case of odd number of colors. This gives a mapping of digraphs
                                                                   $$
                           {\bf SQ}(\langle 2m+1 \rangle,d) \to {\bf Q}(\{0\}\cup [m], d/2),
$$
and results in the commutative diagram
            $$
\begin{CD}
{\bf SQ}(\langle 2m+1 \rangle,d) @>>> {\bf Q}(\{0\}\cup [m],d/2)\\
@VVV @VVV \\
{\bf SQ}(\langle 2m \rangle,d) @>>>       {\bf Q}([m],d/2)
\end{CD}
              $$
where the vertical arrows concern reductions of color  0. 
\medskip

\noindent\textbf{Conjecture 3.} {\em The mappings of digraphs $cor: {\bf SQ}(\langle 2m \rangle, d) \to {\bf Q}(m, d/2)$ and $cor: {\bf SQ} (\langle 2m+1 \rangle, d) \to {\bf Q} (m+1, d/2)$ are surjective, have connected fibers, and are full.}\medskip

Arguing in a spirit of the proof of Proposition~\ref{pro3}, one can check that this conjecture is valid for $d=2$.\medskip

So far, we have assumed that the dimension $d$ is even. However, in the case of $d$ odd (but with an even number of colors), it is also possible to define, in a natural way,  a mapping of digraphs
$$
\textbf{SQ}(\langle 2m \rangle, d) \to \textbf{Q}([m], (d-1)/2).
$$
This is done mostly by arguing as above, and we omit a proof here. The axial space is now generated by the base vectors $e_{d-1},\ldots, e_2$ and has the dimension $(d-1)/2$. One shows that the axis $X $ is a cyclic zonotope with the guiding vectors $\theta_i=\xi_i-\xi_ {- i}$, $i=1,\ldots, m$. A symmetric cubillage $\mathcal{Q}$ of $Z(\langle 2m \rangle,d)$ determines the cubillage $\mathcal{Q}\cap X$ in $X$, and the standard cubillage turns into the standard one.


\section{Miscellaneous}\label{sec10}

Due to Manin and Shekhtman \cite{MS},  the Bruhat poset $B (n,d+1)$ (of type A) can be represented as a factorization from the set of maximal chains in $B (n, d)$. It is natural to ask a similar question for the Bruhat of types  $B$ and $C$. However, in this way, we encounter some difficulties.

First, since Conjecture 1 (in Sect.~\ref{sec5}) has not been verified in general, we do not know whether any maximal chain of symmetric cubillages in $Z(n,d)$ begins with the standard cubillage.

Second, even if a maximal chain concerning $Z(n, d)$ begins with the standard cubillage and ends with the anti-standard one, in view of the presence of barrel flip arrows, the chain of corresponding membranes in the zonotope $Z(n, d+1)$ need not determine exactly a cubillage, but merely a cubillage with possible ``barrel holes''.

These two troubles do not happen if $n$ is even and $d$ is odd. However, even in this case, there is a third (and the main) obstacle: a maximum chain of symmetric cubillages in ${\bf SQ}(n,d)$, being represented as a chain of membranes in $Z(n, d+1)$, though determines a cubillage of  $Z (n, d+1)$, but merely a \emph{skew-symmetric} cubillage.

The latter difficulty can partially be circumvented by observing that (for $n,d$ even) the sets of symmetric and skew-symmetric cubillages of the zonotope $Z (n,d)$ are bijective to each other. Such a bijection is described in~\cite[Sect.~9]{DKK3} explicitly in terms of spectra (vertex sets) of cubillages. An alternative description is as follows. 

In Section~\ref{sec4}, we introduced the diagonal automorphism $A$ of the space ${\mathbb R}^{[d]}$. This automorphism multiplies the base vector $e_j$ by $(-1)^{d+j-1}$; when $d$ is even (that we now assume) the multiplier is $(-1)^{j-1}$. The fixed (axial) space is generated by the vectors $e_{d-1}, e_{d-3},\ldots, e_1$, and the orthogonal one by the vectors $e_d,\ldots, e_2$.

Let us now consider the other automorphism $D$ of the space ${\mathbb R}^{[d]}$ that transfers each base vector $e_j$ to $e_{d+1-j}$. In particular, $e_d$ is transferred to $e_1$, and $e_1$ to $e_d$. This transforms the symmetric cyclic zonotope $Z$ (generated by vectors $\xi _i$ lying on the Veronese curve) into the zonotope $DZ$ generated by the vectors $D\xi _i$. Such a $D\xi _i$ is of the form $(t_i^{d-1},\ldots, t_i, 1)$ and does not lie on the Veronese curve. However, if we divide it by $t_i^{d-1}$ (using the fact that $n$ is even, and therefore all $t_i$'s are nonzero), we obtain the vector $(1,1/t_i,\ldots, (1/t_i)^{d-1})$ already lying on the Veronese curve.

Of course, we should take into account that the points $t_i$ are modified when coming to $1/t_i$. (As before, we assume that $t_i=-t_{n+1-i}$, that is, that the points $t_i$ are located ``symmetrically'' relative to 0.) In particular, the order of colors changes (though this is not so important). An important fact is that the new zonotope $DZ$ is again cyclic and symmetric. And any cubillage $\mathcal Q $ of  $Z$ gives a cubillage $D\mathcal Q $ of $DZ$.

\begin{prop} \label{pro7} 
If $d$ is even, then the map $D$ transforms symmetric cubillages of $Z$ into skew-symmetric ones, and transforms skew-symmetric cubillages into symmetric ones.
\end{prop}

 \begin{proof} 
Observe that the automorphisms $A$ and $D$ anti-commute, $AD= - DA$. Indeed, $ADe_j=Ae_{d+1-j}=(-1)^{d+(d+1-j)-1}e_{d+1-j}=(-1)^je_{d+1-j}$, whereas $DAe_j=D(-1)^{j-1}e_j=(-1)^{j-1}e_{d+1-j}$. Therefore, if a cubillage $ \mathcal Q $ is symmetric, that is, $A\mathcal Q =\mathcal Q $, then
         $$
               A(D\mathcal Q )=-DA\mathcal Q =-D\mathcal Q ,
           $$
and therefore $D\mathcal Q $ is skew-symmetric. The second assertion is shown similarly. 
 \end{proof}

As a consequence,
 
 \begin{numitem1}
for $n$ even and $d$ odd, a maximum chain in ${\bf SQ} (n,d)$ gives a symmetric cubillage in the zonotope $Z (n, d+1)$. 
 \end{numitem1}
 
It is important to emphasize that in this way, we produce all symmetric cubillages of $Z (n, d+1)$. In its turn, applying $D$, we get a skew-symmetric cubillage of  $Z (n, d+1)$.  Note that the chain of symmetric membranes in $Z (n, d+1)$ is actually constructed as in the proof of Proposition \ref{pro1}.
 \medskip

Finally, if $d$ is odd, then $D$ preserves the axis and commutes with $A$. Therefore, for $n$ even and $d$ odd, using $D$, we transform symmetric cubillages into symmetric ones (they concern the different zonotope $DZ$, but using their spectra, we can construct their counterparts in the original zonotope, thus getting one more involution on symmetric cubillages of odd dimensions).

\renewcommand{\refname}{References}

      \end{document}